\documentclass{article}

\usepackage{amsfonts}
\usepackage{amssymb}
\usepackage[english]{babel}
\usepackage{algorithm}
\usepackage{algorithmic}
\usepackage{booktabs}
\usepackage{nicefrac}
\usepackage{authblk}
\usepackage[small]{titlesec}
\usepackage{tikz,pgfplots}

\usepackage[
,breaklinks=true  
,colorlinks       
,linkcolor=blue
,urlcolor=red!50!black    
,citecolor=black!40!green
,bookmarks         
,bookmarksnumbered,
pdftex=true,              
plainpages=false,      
hypertexnames=false,     
pdfpagelabels=true,   
hyperindex=true,
pdfa=true
]{hyperref}
\usepackage[leqno]{amsmath}
\usepackage{amsthm}

\newtheorem{theorem}{Theorem}[section]
\newtheorem{definition}[theorem]{Definition}
\newtheorem{lemma}[theorem]{Lemma}

\newtheoremstyle{mystyle}
{}
{}
{\normalfont}
{}
{\normalfont\itshape}
{.}
{ }
{}

\theoremstyle{mystyle}
\newtheorem{example}[theorem]{Example}
\newtheorem{counterexample}[theorem]{Counterexample}

\numberwithin{equation}{section}

\newcommand{\ttfrac}[2]{{\tiny\frac{#1}{#2}}}

\newcommand{\R}{\mathbb{R}}
\newcommand{\Z}{\mathbb{Z}}
\newcommand{\N}{\mathbb{N}}

\newcommand{\smax}{\sigma_{\max}}
\renewcommand{\d}{\text{d}}

\newcommand{\conv}{\operatorname{conv}}

\newcommand{\ie}{i.e., }
\newcommand{\eg}{e.g., }
\newcommand{\weak}{\rightharpoonup}
\newcommand{\embed}{\hookrightarrow}

\providecommand{\keywords}[1]{\textbf{Keywords.} #1}

\title{\bf Parabolic optimal control
	problems with combinatorial switching constraints \\ Part I:  Convex relaxations\thanks{This work has partially been supported by Deutsche Forschungsgemeinschaft (DFG) under grant no.~BU~2313/7-1 and ME~3281/10-1.}}
\author{Christoph~Buchheim, Alexandra~Gr\"utering and Christian~Meyer\footnote{\{christoph.buchheim,alexandra.gruetering,christian.meyer\}@math.tu-dortmund.de}}
\date{\vspace{-5ex}}
\affil{Department of Mathematics, TU Dortmund University, Germany}

\begin{document}
\maketitle
\renewcommand{\abstractname}{} 
\begin{abstract}
   We consider optimal control problems for partial differential
   equations where the controls take binary values but vary over the
   time horizon, they can thus be seen as dynamic switches.  The
   switching patterns may be subject to combinatorial constraints such
   as, \eg an upper bound on the total number of switchings or a
   lower bound on the time between two switchings. While such
   combinatorial constraints are often seen as an additional
   complication that is treated in a heuristic postprocessing, the core
   of our approach is to investigate the convex hull of all feasible
   switching patterns in order to define a tight convex relaxation of
   the control problem. The convex relaxation is built by cutting
   planes derived from finite-dimensional projections, which can be
   studied by means of polyhedral combinatorics. A numerical example for the case of a bounded number of switchings shows
   that our approach can significantly improve the dual bounds given by the straightforward continuous relaxation, which is obtained by relaxing binarity constraints.
  
  \keywords{PDE-constrained optimization, switching time optimization, convex relaxations}
\end{abstract}

\section{Introduction} \label{sec: intro}

Mixed-integer optimal control of a system governed by partial or ordinary
differential equations became a hot research topic in the last decade,
as a variety of applications leads to such control problems. In
particular, the control often comes in form of a finite set of
switches which can be operated within a given continuous time horizon,
\eg by shifting of gear-switches in automotive
engineering~\cite{GER05,KSB10, SBF13} or by switching of valves or
compressors in gas and water
networks~\cite{FUEG09,HAN20}. Consequently, various approaches are
discussed in the literature to address optimal control problems with
discrete control variables, often known as mixed-integer optimal
control problems (MIOCPs).  Direct methods, based on the
\textit{first-discretize-then-optimize}\/ paradigm, are widely used to
tackle MIOCPs; see for instance~\cite{GER05} and~\cite{VSG20}. The
control and, if desired, the state are discretized in time and space,
in order to approximate the problem by a large, typically non-convex,
finite-dimensional mixed-integer nonlinear programming problem
(MINLP). The latter can be addressed by standard techniques; see
\cite{Lee12} or~\cite{BKLL13} for surveys on algorithms for
MINLPs. However, the size of the arising MINLPs easily becomes too
large to solve them to proven optimality.  In particular, direct
methods are not promising for optimal control problems governed by
partial differential equations~(PDEs)~\cite{GPRS19, SH20}.

In contrast, arbitrary close approximations of MIOCPs can be computed
efficiently by first replacing the set of discrete control values by
its convex hull and then appropriately rounding the result. The most
common approximation methods for systems governed by ordinary
differential equations~(ODEs) are the \textit{Sum-Up Rounding}
strategy~\cite{SA12,KLM20} and the \textit{Next Force Rounding}
strategy~\cite{JUN14}. PDE-constrained problems can also be addressed
with the Sum-Up Ronding strategy~\cite{HS13}. However, in the presence
of additional combinatorial constraints, the latter may be
violated~\cite[Sect.~5.4]{MAN19}, and the heuristics used to obtain
feasible solutions often do not perform well~\cite[Example~3.2]{KML17}. Therefore, when aiming at globally optimal solutions, such approaches may only serve for computing primal bounds.  
To minimize the
integrality error, the \textit{Combinatorial Integral Approximation
	(CIA)}~\cite{SA05} tracks the average of a relaxed solution over a
given rounding grid by a piecewise constant integer control and the
discretized problem is solved by a tailored branch-and-bound
algorithm~\cite{JRS15, KMS11}. The approach was again generalized to
PDE-constrained problems~\cite{HKMS19}. To reduce the (undesired)
chattering behavior of the rounded control the total variation is
constrained~\cite{ZS20} or switching cost aware rounding algorithms
are considered~\cite{BHKM20, BK20}.

Other approaches optimize the switching times, \eg by controlling the
switching times through a continuous time control function which
scales the length of minor time intervals~\cite{GER06, ROL17} or by
including a fixed number of transition times as decision variables
into the MIOCP and solving the corresponding finite-dimensional
non-convex problems by gradient descent techniques~\cite{SOG16,EWA06}
or by second order methods~\cite{JM11,SOG17}.  PDE-constrained optimal
control problems can be addressed by the concept of switching time
optimization as well~\cite{HR16}. Nevertheless, these methods have a
limited applicability, since fairly restrictive assumptions on the
objective and the state dynamics need to be made in order to guarantee
differentiability in the discretized setting~\cite{FMO13}.

In the context of optimal control problems governed by PDEs, switching
constraints are frequently imposed by penalty terms added to the
objective functional~\cite{CIK16, CRKB16, CRK17}.  The arising
penalized problems are non-convex in general and are therefore
convexified by means of the bi-conjugate functional associated with
the penalty term. The desired switching structure of the optimal
solutions of the convexified problems can however only be guaranteed
under additional structural assumptions on the unknown solution.  For
the case of a switching between multiple constant control variables, a
multi-bang approach might be favorable since optimal control problems
subject to box constraints on the control may show a bang-bang
behavior in the absence of a Tikhonov-type regularization term
\cite{Troe79,DH12, CWW18, TW18}.  However, the bang-bang structure of
the optimal control cannot be guaranteed in general.  In order to
promote that the control attains the desired constant values,
$L^0$-penalty terms or suitable indicator functionals are added to the
objective and convex relaxations of the penalty terms based on the
bi-conjugate functional are employed to make the problem amenable for
optimization algorithms~\cite{CK14,CTW18}.  Again, as in case of the
penalization of the switching constraints mentioned above, the
multi-bang structure of the optimal solutions of the convexified
problems can only be ensured under additional assumptions that cannot
be verified a priori. In~\cite{CK16}, the convexification of the
$L^0$-penalty by means of the bi-conjugate functional is employed in
the context of topology optimization, in~\cite{CKK18}, the
$L^0$-penalty is enriched by the BV-seminorm.  $L^0$-penalization
techniques that go without regularization or convexification are for
instance addressed in~\cite{CW20} from a theoretical perspective and
in~\cite{Wac19} with regard to algorithms. However, to the best of
our knowledge, additional combinatorial constraints on the switching
structure have not yet been included in the penalization framework.

In summary, the design of global solvers for MIOCPs with dynamic switches and combinatorial switching
constraints is an open field of research.
The core of our new approach for addressing such problems
is the computation of lower bounds by a tailored convexification of
the set of feasible switching patterns in function space. A counterexample given
in~Section~\ref{sec: convhull} shows that, even when the combinatorial constraint only consists in an upper bound on the total number of switchings, the naive approach of
just relaxing the binarity constraint does not lead to the convex hull
of the set of feasible switching patterns.  Our aim is to determine tighter approximations of this
convex hull by considering finite-dimensional projections that allow
for the efficient computation of cutting planes. Based on the
resulting outer description of the convex hull, in the companion paper~\cite{partII} we develop a tailored outer approximation algorithm  which
converges to a global minimizer of the convex relaxations. The
resulting lower bounds could be used, \eg in a branch-and-bound scheme
to obtain globally optimal solutions of the control problems.

The remainder of this paper is organized as follows. In Section~\ref{sec:
	opt}, we specify the prototypical optimal control problem as well as the class of combinatorial switching constraints considered in this work and show
that the problem admits an optimal solution. In Section~\ref{sec: convhull}, we investigate the convex hull
of feasible switching patterns and show that it can be fully described by cutting planes
lifted from finite-dimensional projections. 
An example in Section~\ref{sec: qbounds} shows the strength of the lower bounds resulting from our tailored convexification. 

\section{Optimal control problem} \label{sec: opt}
For the sake of simplicity, throughout this
paper, we restrict ourselves to a parabolic binary optimal control
problem with switching constraints of the following form:
\begin{equation}\tag{P}\label{eq:optprob}
\left\{\quad
\begin{aligned}
\text{min} \quad & J(y,u) = \tfrac{1}{2}\, \|y - y_{\textup{d}}\|_{L^2(Q)}^2 + \tfrac{\alpha}{2}\,\|u-\tfrac 12\|^2_{L^2(0,T; \R^n)}\\
\text{s.t.} \quad & 
\begin{aligned}[t]
\partial_t y(t,x) - \Delta y(t,x) &= \sum_{j=1}^n u_j(t) \,\psi_j(x) & & \text{in } Q := \Omega \times (0,T),\\
y(t,x) &= 0 & & \text{on } \Gamma := \partial\Omega \times (0,T),\\  
y(0,x) &= y_0(x) & & \text{in } \Omega,
\end{aligned}\\
\text{and} \quad & u \in D.
\end{aligned}
\quad \right.
\end{equation}
Herein, $T > 0$ is a given final time and $\Omega\subset \R^d$, $d\in \N$, denotes a bounded domain, where a
domain is an open and connected subset of a finite-dimensional vector
space, with Lipschitz boundary $\partial \Omega$ in the sense of
\cite[Def. 1.2.2.1]{GRIS85}. The
form functions $\psi_j\in H^{-1}(\Omega)$, $j=1,\dots,n$, as well as
the initial state~$y_0 \in L^2(\Omega)$ are given. Moreover,
\[
D \subset \big\{ u \in BV(0,T;\R^n)\colon u(t) \in \{0,1\}^n \text{ f.a.a.\ } t \in (0,T)\big\}
\]
denotes the set of feasible \emph{switching controls}. Finally,
$y_{\textup{d}} \in L^2(Q)$ is a given desired state
and $\alpha \geq 0$ is a Tikhonov parameter weighting the mean
deviation from~$\tfrac 12$. Note that the choice of~$\alpha$ does not
have any impact on the set of optimal solutions of~\eqref{eq:optprob},
as $u \in \{0,1\}^n \text{ a.e.~in } (0,T)$ and hence the Tikhonov
term is constant. However, the convex relaxations of~\eqref{eq:optprob}
considered in this paper as well as their optimal values are
influenced by~$\alpha$.

The particular challenge of our problem are the combinatorial
switching constraints modeled by the
set~$D$
%% \[D\subset \big\{ u \in BV(0,T;\R^n)\colon u(t)
%% \in \{0,1\}^n \text{ f.a.a.\ } t \in (0,T)\big\}\]
of
feasible controls.  It is supposed to satisfy the two
following assumptions:
\begin{align}
& \text{$D$ is a bounded set in $BV(0,T;\R^n)$,} \tag{D1}\label{eq:D1} \\
& \text{$D$ is closed in $L^p(0,T;\R^n)$ for some fixed $p \in [1,\infty)$.} \tag{D2}\label{eq:D2}
\end{align}
Here, $BV(0,T;\R^n)$ denotes the set of all vector-valued functions with bounded variation, \ie
\[BV(0,T;\R^n):=\{u\in L^1(0,T;\R^n): u_i\in BV(0,T) \text{ for } i=1,\ldots,n\,
\}\] equipped with the norm
\[\|u\|_{BV(0,T;\R^n)} := \|u\|_{L^1(0,T;\R^n)} + \sum_{j=1}^n |u_j|_{BV(0,T)}\;.\]
For more details on the space of bounded variation functions, see, \eg
\cite[Chap.~10]{ATT14}. Note that, in our case, the
BV-seminorm~$|u_j|_{BV(0,T)}$ agrees with the minimal number of
switchings of any representative of~$u_j$ with values
in~$\{0,1\}$. 

A possible example for such a set is
\begin{equation}\label{eq:Dex}
\begin{aligned}
D_{\max} := \big\{ u \in BV(0,T;\R^n)\colon \;
& u(t) \in \{0,1\}^n \text{ f.a.a.\ } t \in (0,T),\\
& |u_j|_{BV(0,T)} \leq \smax\; \forall \,j = 1, \dots, n \big\},
\end{aligned}
\end{equation}
where $|\cdot|_{BV(0,T)}$ denotes the BV-seminorm and $\smax\in \N$
is a given number. The set $D_{\max}$ meets the assumptions~\eqref{eq:D1}
and~\eqref{eq:D2}, as we will show in Example \ref{ex: boundedshift}. This choice of~$D$ is motivated by the following
application-driven scenario: suppose $y$ is the temperature of a body
covering the domain $\Omega$ and the aim of the optimization is to
minimize the deviation of $y$ from a given desired state
$y_{\textup{d}}$, by means of $n$ given heat sources modeled by the
form functions~$\psi_j$, $j=1, \dots, n$. These heat sources can be
switched on and off at arbitrary points in time, but we are only
allowed to shift each switch for at most $\smax$ times.  This leads to
the set $D_{\max}$.

Various other practically relevant choices of~$D$ are conceivable. For
instance, it could be required to bound the time interval between two
shiftings of the same switch from below because of technical
limitations; this kind of restriction is known as minimum dwell time
constraints in the optimal control community and as min-up/min-down
constraints in the unit commitment community. See Example \ref{ex: combswpoint} for a discussion and generalization of this class of constraints. Another condition may be
that certain switches are not allowed to be used (or switched on) at
the same time.

Our previous assumptions guarantee that the PDE contained in~\eqref{eq:optprob} admits a unique weak solution $y\in W(0,T) :=
H^1(0,T;H^{-1}(\Omega)) \cap L^2(0,T; H^1_0(\Omega))$ for every $u\in D\subset~L^2(0,T;\R^n)$; see
\cite[Chapter~3]{Troe10}. The associated solution operator~$S\colon
L^2(0,T;\R^n) \to W(0,T)$ is affine and continuous. Using this
solution operator, the problem \eqref{eq:optprob} can be written as
\begin{equation} \tag{P'} \label{eq:P'}
\left\{\quad
\begin{aligned}
\mbox{min }~ &f(u)=J(Su,u) \\ 
\mbox{s.t. }~
& u\in D\;.
\end{aligned}\, \right.
\end{equation}
Note that the objective function $f\colon L^2(0,T;\R^n) \to \R$ is
weakly lower semi-continuous because both~$u \mapsto \|Su -
y_{\textup{d}}\|_{L^2(Q)}^2$ and~$u \mapsto \|u-\tfrac 12 \|_{L^2(0,T;\R^n)}^2$
are convex and lower semi-continuous, thus weakly lower
semi-continuous, and the solution operator $S$ is affine and
continuous, thus weakly continuous.

\begin{theorem}\label{thm:discr_ctrl}
	Let $D\neq \emptyset$. Then Problem~\eqref{eq:P'} admits a global minimizer.
\end{theorem}
\begin{proof}
	Since~$D\neq\emptyset$, we have $f^\star:=\inf_{u\in D}f(u)\in
	\R\cup\{-\infty\}$.  Let $\{u^k\}_{k\in \mathbb{N}}$ in $D$ be an
	infimal sequence with
	\[
	\lim\limits_{k\to \infty} f(u^k) = f^\star\;.
	\]
	We know that $\{u^k\}_{k\in \N}$ is a bounded sequence
	in~$BV(0,T;\R^n)$, since $D$ is a bounded set in~$BV(0,T;\R^n)$ by
	assumption \eqref{eq:D1}, \ie
	\[\sup_{k\in \N}~\Vert u^k \Vert_{BV(0,T;\R^n)} = \sup_{k\in \N} \left(\|u^k\|_{L^1(0,T;\R^n)} + \textstyle\sum_{j=1}^n |u_j^k|_{BV(0,T)}\right) < \infty\;.\]
	By Theorem 10.1.3 and Theorem 10.1.4 in~\cite{ATT14}, $BV(0,T;\R^n)$
	is compactly embedded in $L^p(0,T;\R^n)$, and hence there exists a
	strongly convergent subsequence, which we again denote by
	$\{u^{k}\}_{k\in \N}$, such that $ u^k \to u^\star\in
	L^{p}(0,T;\R^n) \text{ for } k\to \infty.$ 
	Since $D$ is closed in $L^p(0,T;\R^n)$ by
	condition \eqref{eq:D2}, we deduce that $u^\star\in D$. The weak lower
	semi-continuity of the objective function~$f$ leads to
	\[f(u^\star)\leq \liminf\limits_{k\to\infty } f(u^k)=f^\star\;.\]
	This implies $f^\star  > -\infty$ as well as the optimality of $u^\star$ for \eqref{eq:P'}.
\end{proof}

\section{Convex hull description} \label{sec: convhull}

The crucial ingredient of our approach is the outer description of the
convex hull of the set~$D$ of feasible switching patterns by linear
inequalities. In general, just replacing $\{0,1\}$ with $[0,1]$ in the
definition of~$D$ does not lead to the convex hull of~$D$ in any
$L^p$-space. This is true even in the case of just one switch that can
be changed at most once on the entire time horizon, \ie if the
feasible switching control is required to belong to
\begin{equation}\label{eq:counter}
D:=\{ u \in BV(0,T)\colon \;
u(t) \in \{0,1\} \text{ f.a.a.\ } t \in (0,T),
|u|_{BV(0,T)} \leq 1\}\;.
\end{equation}
Essentially, the naive approach does not consider the monotonicity of the
switches in~$D$, as we will see in the following counterexample.

\begin{counterexample}
	Let~$D$ be defined as in~\eqref{eq:counter} and consider the function
	\[
	u(t):=\begin{cases}
	\tfrac{1}{2} & \text{if }t \in [\tfrac{1}{3}T,\tfrac{2}{3}T ] \\
	0 & \text{otherwise}.
	\end{cases}\]
	Obviously, we have $u\in BV(0,T)$ with $u(t) \in [0,1]$ for
	$t\in(0,T)$ and $|u|_{BV(0,T)}=1$. However, we claim that~$u$
	does not belong to the closed convex hull of~$D$ in~$L^p(0,T)$ for
	any $p\in [1,\infty)$.
	
	Assume on contrary that $u\in\overline{\conv(D)}^{L^p(0,T)}$
	for some $p\in [1, \infty)$. Then there exists a sequence
	$\{u^k\}_{k\in \N}\subset \conv(D)$ with $u^k \to u$ in $L^p(0,T)$
	for $k\to \infty$. In particular, $\{u^k\}_{k\in \N}$ converges
	strongly to $u$ in $L^1(0,T)$ due to $L^p(0,T)\embed
	L^1(0,T)$, \ie
	\[
	\int_{0}^{T} |u^k-u|\, \d t \to 0\mbox{ for } k\to \infty\;.
	\]
	Define $A^k:= \{t\in[\tfrac{1}{3}T,\tfrac{2}{3}T ]: u^k(t)\geq
	\tfrac{2}{5}\}$. We claim that there exists $k_0\in \N$ such that
	the sets~$A^k$, $k\geq k_0$, have a positive
	Lebesgue-measure. Indeed, if such a $k_0\in \N$ did not exist, then
	we could find a subsequence, which we denote by the same symbol
	$\{A^k\}_{k\in\N}$ for simplicity, such that $\lambda(A^k)=0$ for
	all $k\in \N$, where $\lambda(A^k)$ denotes the Lebesgue measure of
	$A^k$.  With $\lambda(A^k)=0$, it follows
	\[ \int_{0}^{T} |u^k-u|\, \d t \geq
	\int_{\ttfrac{1}{3}T}^{\ttfrac{2}{3}T} |u^k-\tfrac{1}{2}|\, \d t =
	\int_{[\ttfrac{1}{3}T,\ttfrac{2}{3}T]\setminus A^k}
	|u^k-\tfrac{1}{2}|\, \d t >\tfrac{1}{30}T\;,
	\]
	where the last inequality holds due to $ |u^k-\tfrac{1}{2}|>
	\tfrac{1}{10}$ for all $t \in [\tfrac{1}{3}T,\tfrac{2}{3}T
	]\setminus A^k$ by definition of~$A^k$. This contradicts the strong
	convergence of $u^k$ to $u$ in $L^1(0,T)$. Thus, a number $k_0\in
	\N$ exists with $\lambda(A^k)>0$ for all $k\geq k_0$.
	
	Now, let $k\geq k_0$ be arbitrary. We write $u^k\in \conv(D)$ as a convex combination
	\[u^k=\sum_{l=1}^{m_k} \mu_l^k y_l^k\]
	of functions in $D$. Let
	$t_0\in A^k$ be a Lebesgue point of all functions $y_l^k\in D$,
	$1\leq l\leq m_k$, which exists since the set of all non-Lebesgue
	points of $y_l^k$ is a set of Lebesgue measure zero. Then, we know
	\[
	\tfrac{2}{5}\leq u^k(t_0)=\sum_{l=1}^{m_k} \mu_l^k y_l^k(t_0)\;.
	\]
	Set $I_k:=\{l\in\{1,\ldots,m_k\}: y_l^k(t_0)=1\}$. The
	inequality then implies
	\begin{equation}\label{eq:ineq}
	\tfrac{2}{5}\leq \sum_{l\in I_k} \mu_l^k\;.
	\end{equation}
	Since $y_l^k(t_0)=1$ for $l\in I_k$ and $y_l^k$ might shift at most once due to $|y_l^k|_{BV(0,T)}\leq 1$, we deduce that either $y_l^k$ was first turned off and then turned on in $(0,t_0)$, such that $y_l^k(t)\equiv 1$ a.e. in $(t_0,T)$ holds, or $y_l^k$ was first turned on, \ie $y_l^k(t)\equiv 1$ a.e. in $(0,t_0)$. Consequently, we get $y_l^k(t)\equiv 1$ a.e. in $(0,\tfrac{1}{3}T )$ or $(\tfrac{2}{3}T,T )$ for every $l\in I_k$. The latter, together with
	\eqref{eq:ineq}, yields
	\[\int_{0}^{T} |u^k-u|\ \d t \geq \int_{(0,\ttfrac{1}{3}T)\cup (\ttfrac{2}{3}T,T)} |u^k|\, \d t  
	\geq\sum_{l\in I_k} \mu_l^k \int_{(0,\ttfrac{1}{3}T)\cup (\ttfrac{2}{3}T,T)}y_l^k\, \d t  
	\geq \tfrac{2}{15} T\;,
	\]
	which contradicts the strong convergence of $u^k$ to $u$ in
	$L^1(0,T)$.
\end{counterexample}

This counterexample shows that we cannot expect to obtain a tight
description of~$\conv(D)$ without a closer investigation of the
specific switching constraint under consideration. Our basic idea is
to reduce this investigation to a purely combinatorial task by
projecting the set~$D$ to finite-dimensional spaces~$\R^M$, by means
of~$M\in \N$ linear and continuous functionals $\Phi_i \in
L^p(0,T;\R^n)^*$, $i=1, \dots, M$. In the following, we restrict
ourselves to local averaging operators of the
form \begin{equation}\label{eq:localaveraging} \langle \Phi_{(j-1)N+i}, u
\rangle := \tfrac{1}{\lambda(I_{i})}\int_{I_{i}} u_{j}\,\d t
\end{equation}
for $j=1,\ldots,n$ with suitably chosen subintervals $I_i\subset (0,T)$, $i=1,\ldots,N$, and $M:=n\,N$. 
The resulting projection then reads
\begin{equation}\label{eq:pi}
\Pi\colon BV(0,T;\R^n) \ni u \mapsto \big(\langle \Phi_l, u
\rangle\big)^M_{l=1} \in \R^{M}\;.
\end{equation}
Note that~$\Pi$ is a linear mapping. The core result underlying our approach
is that, for increasing~$N$, projections~$\Pi_N$ can be designed such that
\begin{equation}\label{eq:convapprox}
\overline{\conv(D)}^{L^p(0,T;\R^n)} =
\bigcap_{N\in \N} \{ v \in L^p(0,T;\R^n)\colon \Pi_N(v) \in
C_{D,\Pi_N}\}
\end{equation}
where
\[
C_{D,\Pi} := \conv\{\Pi (u)\colon u \in D \}\subset \R^{M}\;.
\]
In other words, an outer description of all finite-dimensional convex hulls
$C_{D,\Pi}$ also leads to an outer description of the convex hull of
$D$ in function space.

We first observe that our general assumptions \eqref{eq:D1} and
\eqref{eq:D2} guarantee the closedness of the finite-dimensional set
$C_{D,\Pi}$ in $\R^M$.
\begin{lemma}\label{lem: finiteclosed}
	For any\/ $\Pi$ as in~\eqref{eq:pi}, the set $C_{D,\Pi}$ is closed in $\R^M$.
\end{lemma}
\begin{proof}
	Let $\{\Pi(u^k)\}_{k\in \N}\subset \R^M$ be a convergent sequence
	in $\Pi(D)$, resulting from the projection of feasible switching
	controls $u^k\in D$ for $k\in \N$, with $\Pi(u^k) \to \omega $ in
	$\R^M$. The sequence $\{u^k\}_{k\in \N}\subset D$ is bounded in
	$BV(0,T;\R^n)$ by~\eqref{eq:D1}. As in~Theorem~\ref{thm:discr_ctrl}, the
	compactness of the embedding $BV(0,T;\R^n)\embed L^p(0,T;\R^n)$ by
	Theorem~10.1.3 and Theorem~10.1.4 in~\cite{ATT14} implies the existence of a
	strongly convergent subsequence, again denoted by
	$\{u^{k}\}_{k\in \N}$, such that $ u^k \to u\in L^{p}(0,T;\R^n)
	\text{ for } k\to \infty.$ Since $D$ is closed in $L^p(0,T;\R^n)$
	by~\eqref{eq:D2}, we deduce $u\in D$.  By continuity of $\Pi$ in
	$L^p(0,T;\R^n)$, we then have
	\[
	\omega = \lim\limits_{k\to \infty}\Pi (u^k) =\Pi(u)
	\] so that $\omega$ lies in $\Pi(D).$ Hence, the set $\Pi(D)$ is closed in $\R^M$. It is also bounded, thus compact, such that $C_{D,\Pi}$ is closed as the convex hull of a compact set in $\R^M$.
\end{proof}

As a consequence, we obtain that the subset of $L^p(0,T;\R^n)$
corresponding to the finite-dimensional projection $\Pi$ is convex and
closed in $L^p(0,T;\R^n)$.
\begin{lemma}\label{lem: closed}
	For any\/ $\Pi$ as in~\eqref{eq:pi}, the set $\{ v \in
	L^p(0,T;\R^n)\colon\Pi(v) \in C_{D,\Pi}\}$ is convex and closed
	in~$L^p(0,T;\R^n)$.
\end{lemma}
\begin{proof}
	The convexity assertion follows from the convexity of $C_{D,\Pi}$
	together with the linearity of~$\Pi$. Closedness follows from
	Lemma~\ref{lem: finiteclosed} and the continuity of~$\Pi$
	in~$L^p(0,T;\R^n)$.
\end{proof} 

By the following observation, each projection~$\Pi$ gives rise to a
relaxation of the closed convex hull of~$D$ in~$L^p(0,T;\R^n)$. These relaxations can be used to derive outer approximations by
linear inequalities.
\begin{lemma}\label{lem: convD}
	For any\/ $\Pi$ as in~\eqref{eq:pi}, we have
	\[\overline{\conv (D)}^{ L^p(0,T;\R^n)}\subseteq \{ v \in
	L^p(0,T;\R^n)\colon \Pi(v) \in C_{D,\Pi}\}=:V\;.\]
\end{lemma}
\begin{proof}
	By construction of $C_{D,\Pi}$, every $u\in D$ satisfies $\Pi(u)
	\in C_{D,\Pi}$. The linearity of~$\Pi$ leads to
	$\conv (D)\subset V$, 
	using the convexity of~$V$ stated in Lemma~\ref{lem:
		closed}.  Again by Lemma~\ref{lem: closed}, the set~$V$ is closed
	in $L^p(0,T;\R^n)$, which shows the desired result.
\end{proof}

The following result shows that the convex hull of the set of feasible
switching controls can be fully described with the help of appropriate
finite-dimensional sets $C_{D,\Pi}$. 
With a little abuse of notation, we slightly change the notation of the local 
averaging operators in the sense that the number of subintervals now differs 
from the dimension $M$ of the range of $\Pi$, see \eqref{eq:Pik} below, 
in order to ease the proof of the following theorem.

\begin{theorem}\label{thm: convD}
	For each~$k\in\N$, let~$I^k_1,\dots,I_{N_k}^k$, $N_k\in\N$, be
	disjoint open intervals in~$(0,T)$ such that 
	\begin{itemize}
		\item[(i)] $\bigcup_{i=1}^{N_k} \overline{I_i^k} = [0,T]$ for all $k\in\N$ and
		\item[(ii)]$\max_{i=1,\dots,N_k}\lambda(I_i^k)\to 0\; \mbox{ for } k\to \infty$.
	\end{itemize}
	Set $M_k := n\, N_k$ and define projections~$\Pi_{k}\colon BV(0,T;\R^n) \to \R^{M_k}$,
	for~$k\in\N$, by 
	\begin{equation}\label{eq:Pik}
	\langle\Phi_{(j-1)N_k+i}^k,u\rangle :=
	\tfrac{1}{\lambda (I_i^k)}\int_{I_i^k} u_j(t)\, \d t\;
	\end{equation}
	for $j=1,\ldots,n$ and $i=1,\ldots,N_k$. Moreover, set
	\[V_k:=\{ v \in L^p(0,T;\R^n)\colon \Pi_{k}(v) \in
	C_{D,\Pi_k}\}\;.\]
	Then
	\begin{equation}\label{eq:convD}
	\overline{\conv (D)}^{ L^p(0,T;\R^n)} = \bigcap_{k\in \N} V_k\;.
	\end{equation}
	
\end{theorem}

\begin{proof}
	The inclusion \grqq$\subseteq$\grqq\ in \eqref{eq:convD} follows directly from
	Lemma~\ref{lem: convD}, it thus remains to show \grqq$\supseteq$\grqq. For this,
	let  
	\[u\in\bigcap_{k\in \N} V_k\;.
	\]    
	By definition of $u$, we have $\Pi_{k}(u)\in C_{D,\Pi_{k}}$. Hence, there exist~$v_l^k \in D$ for $l=1,\ldots,m$, where $m=m(k)\in\N$ may depend on $k$, as well coefficients $\mu^k_l \geq 0$ with  $\sum_{l=1}^m \mu^k_l=1$ and 
	\[
	\Pi_{k}(u) = \sum_{l=1}^m \mu^k_l\, \Pi_{k}(v_l^k)\;.
	\] Set $u^k := \sum_{l=1}^m \mu_l^k v_l^k \in \conv(D)$. By construction and the linearity of the projection, we have $\Pi_k(u^k)=\Pi_k(u)$, \ie 
	\begin{equation}\label{eq:PiNeq}
	\int_{I_i^k} (u^k - u)\,\d t = 0\quad \forall\, i=1,\ldots,N_k, \;
	k\in \N\;.
	\end{equation}
	Let $k\in \N$ be fixed. Thanks to  assumption (i), we conclude that for every $\ell\in \N$ it holds 
	\[
	\lambda \Big(I_i^\ell \setminus \bigcup_{I_r^k\subset I_i^\ell} I_r^k\Big)\leq 2\max_{r=1,\dots,N_k} \lambda(I_r^k)\;.
	\]
	Set $E^l_i:=\bigcup_{I_r^k\subset I_i^\ell} I_r^k$ for all $\ell \in \N$ and $i=1,\ldots,N_\ell$.
	Then \eqref{eq:PiNeq} implies 
	\begin{equation*}
	\begin{aligned}
	\int_{I_i^\ell} (u^k - u)\, \d t&=  \int_{I_i^\ell\setminus E^l_i} (u^k - u)\, \d t + \int_{E^l_i} (u^k - u)\, \d t\ \\ &=  \int_{I_i^\ell\setminus E^l_i} (u^k - u)\, \d t
	\end{aligned}
	\end{equation*} 
	and thus
	\begin{equation}\label{eq:PiLeq}
	\begin{aligned}
	\Big| \int_{I_i^\ell} (u^k - u)\, \d t\Big|\leq\int_{I_i^\ell\setminus E^l_i} |u^k - &u|\, \d t \leq \lambda(I_i^\ell\setminus E^l_i) \\&\leq 2 \max_{r=1,\dots,N_k} \lambda(I_r^k) \quad \forall i=1,\ldots,N_\ell,\ \ell \in \N
	\end{aligned}
	\end{equation}
	Since $u^k(t)\in[0,1]^n$ holds almost everywhere in $(0,T)$, there
	exists a weakly convergent subsequence, which we denote by the same
	symbol for simplicity, with $u^k \weak \tilde{u}$ in~$
	L^p(0,T;\R^n)$. Together with~\eqref{eq:PiLeq} and $\max_{r=1,\dots,N_k}\lambda(I_r^k)\to 0\; \mbox{ for }
	k\to \infty$, the weak convergence
	of $\{u^k\}_{k\in\N}$ to $\tilde{u} $ implies
	\begin{equation}\label{eq:wutildeueq}
	\int_{I_i^\ell} (\tilde u - u)\, \d t = 0\quad \forall\, i=1,\ldots,N_\ell, \; \ell\in \N\;.
	\end{equation} 
	It is well known that the span of the characteristic functions $\chi_{I_i^\ell}$, $i=1,\ldots,N_\ell$ $\ell\in\mathbb{N}$, is dense in $L^p(0,T)$, so that \eqref{eq:wutildeueq} immediately yields $u=\tilde{u}$ in $L^p(0,T;\R^n).$ We thus obtain $u^k
	\weak u$ in $L^p(0,T;\R^n)$. The set $\overline{\conv D}^{L^p(0,T;\R^n)}$ is convex and closed, thus weakly closed, so that we deduce $u\in\overline{\conv D}^{L^p(0,T;\R^n)}$.
\end{proof}

Our aim is to exploit the result of Theorem~\ref{thm: convD} in order to
obtain outer descriptions of the convex hull of~$D$ in function space
from outer descriptions of finite-dimensional sets of the
form~$C_{D,\Pi}$. This approach is particularly appealing in case~$C_{D,\Pi}$
is a polyhedron. Before discussing some relevant classes of
constraints where this holds true, we first show that polyhedricity cannot be
guaranteed in general. In fact, the following construction shows that
every closed convex set~$K\subseteq[0,1]^M$ can arise as~$C_{D,\Pi}$
for some feasible set~$D$.
\begin{example}\label{ex: nonpoly}
	Let~$M\in\N$ and~$K\subseteq[0,1]^M$ be a closed convex set. Define~$T=M$ and
	\[
	\begin{aligned}
	D_{K} := \big\{ u \in BV(0,T)\colon \;
	& u(t) \in \{0,1\} \text{ f.a.a.\ } t \in (0,T),\\
	& |u|_{BV(0,T)} \le M,\; \textstyle(\int_{i-1}^{i}u\, \d t)_{i=1}^M\in K \big\}\;.
	\end{aligned}
	\]
	By definition, the set~$D_K$ satisfies Assumption~\eqref{eq:D1}. Also
	Assumption~\eqref{eq:D2} is easy to verify for arbitrary
	$p\in[1,\infty)$, using the closedness of $K$ and Proposition
	10.1.1(i) in~\cite{ATT14}, which guarantees, for any
	sequence~$\{u^k\}_{k\in \N}\subset D_K$ converging to some $u$
	in~$L^p(0,T)\embed L^1(0,T)$, that
	\[
	|u|_{BV(0,T)}  \leq  \liminf\limits_{k\to \infty}|u^k|_{BV(0,T)} \leq M\;.
	\]
	Defining the projection~$\Pi$ by local averaging on the
	intervals~$(i-1,i)$, $i=1,\dots,M$, we obtain~$\Pi(D_K)=K$ and
	hence, due to convexity of~$K$, we have~$K=C_{D_K,\Pi}$.
\end{example}

In the following subsections, we discuss two of the practically most
relevant classes of constraints~$D$ and investigate the associated
sets~$C_{D,\Pi}$. The first class includes~$D_{\max}$ as defined
in~\eqref{eq:Dex}, whereas the second class includes the minimum dwell
time constraints mentioned in the introduction.  For the remainder of
this section, we always assume that the intervals defining the
projection~$\Pi$ are pairwise disjoint.

\subsection{Pointwise combinatorial constraints}\label{ex: boundedshift}

By Assumption~\eqref{eq:D1}, the total number of shiftings of all
switches is bounded by some~$\sigma\in\N$. A relevant class of
constraints arises when the switches must additionally satisfy certain
combinatorial conditions at any point in time. As an example, it might
be required that two specific switches are never used at the same
time, or that some switch can only be used when another switch is also
used, \eg because they are connected in series.
More formally, we assume that a set~$U\subseteq\{0,1\}^n$ is given and consider the constraint
\[
D_{\max}^\Sigma(U) := \Big\{ u \in BV(0,T;\R^n)\colon \;
u(t) \in U \text{ f.a.a.\ } t \in (0,T),
\sum_{j=1}^n|u_j|_{BV(0,T)} \leq \smax\Big\}.
\]

\begin{lemma}
	The set $D_{\max}^\Sigma(U)$ satisfies Assumptions~\eqref{eq:D1} and~\eqref{eq:D2}.
\end{lemma}

\begin{proof}
	The set $D_{\max}^\Sigma(U)$ obviously satisfies
	\eqref{eq:D1}. Moreover, for any~$p\in[1,\infty)$,
	Proposition~10.1.1(i) in~\cite{ATT14} again guarantees for any
	sequence of controls $\{u^k\}_{k\in \N}\subset D_{\max}^\Sigma(U)$ in~$L^p(0,T;\R^n)\embed
	L^1(0,T;\R^n)$ that converges to some $u$ that
	\[
	|u_j|_{BV(0,T)}  \leq  \liminf\limits_{k\to \infty}|u_j^k|_{BV(0,T)} \leq \smax
	\]
	for $j=1,\ldots,n$, because of $\sup_{k\in \N} |u_j^k|_{BV(0,T)} \leq
	\smax$. Furthermore, since convergence in $L^p(0,T;\R^n)$ implies pointwise almost everywhere 
	convergence for a subsequence, the limit also satisfies $u(t) \in U$ f.a.a.\ $t\in (0,T)$.
	It follows that $D_{\max}^\Sigma(U)$ is closed in
	$L^p(0,T;\R^n)$ and thus fulfills \eqref{eq:D2}. 
\end{proof}

We now show that the projections~$\Pi$ defined in~\eqref{eq:pi} not
only lead to polytopes when applied to~$D_{\max}^\Sigma(U)$, but even
yield integer polytopes, \ie polytopes with integer vertices only.
\begin{theorem}\label{thm: dmax01}
	For any\/ $\Pi$ as in~\eqref{eq:pi}, the
	set~$C_{D_{\max}^\Sigma(U),\Pi}$ is a 0/1-polytope in~$\R^M$.
\end{theorem}
\begin{proof}
	We claim that~$C_{D_{\max}^\Sigma(U),\Pi}=\conv(K)$, where
	\[
	\begin{aligned}
	K:=\{\Pi(u) \colon & u \in D_{\max}^\Sigma(U) \text{ and for all }i=1,\ldots,M \text{ there exists } w_i\in U \\
	& \text{with } u(t)\equiv w_i \text{ f.a.a.\ } t\in I_i\}\;.
	\end{aligned}
	\]
	From this, the result follows directly, as~$K\subseteq\{0,1\}^M$
	holds by definition.
	
	The direction \grqq$\supseteq$\grqq\ is trivial, since $K$ is a
	subset of $\{\Pi(u) \colon u \in D_{\max}^\Sigma(U)\}$. It thus
	remains to show \grqq$\subseteq$\grqq. For this, let $u\in
	D_{\max}^\Sigma(U)$. We need to show that $\Pi(u)$ can be written as
	a convex combination of vectors in~$K$. Let $m\in\{0,\dots,M\}$ denote the
	number of intervals in which at least one of the switches is
	shifted in~$u$. We prove the assertion by means of complete induction over
	the number $m$. For $m=0$, we clearly have $\Pi(u)\in K\subseteq \conv(K)$.
	
	So let the number of intervals in which at least one of the switches
	is shifted be $m+1$. Additionally, let $\ell\in\{1,\ldots,M\}$ be an
	index so that at least one switch is shifted in the interval
	$I_\ell$. Since we have the upper bound $\smax$ on the total number
	of shiftings, only finitely many shiftings can be in the interval
	$I_\ell$. Hence, $I_\ell$ can be divided into disjoint subintervals
	$I^1_\ell,\ldots,I^s_\ell$ such that
	$\overline{I_\ell}=\bigcup_{k=1}^{s} \overline{I^k_\ell}$ and there
	exist $w_k\in U$ with $u(t)=w_k$ f.a.a.\ $t\in I^k_\ell$, $1\leq k
	\leq s$. Define functions $u^k$ for $k=1,\ldots,s$ as follows:
	\[u^k(t):=\begin{cases}
	w_k & \text{if }t\in I_\ell \\ 
	u(t) & \text{otherwise}\;. 
	\end{cases}
	\]
	Due to $u\in U$ a.e.~in $(0,T)$ and $w_k\in U$, $u^k(t)$ is
	a vector in~$U$ f.a.a.\ $t\in(0,T)$ and  for
	$k=1,\ldots,s$. Furthermore, $u^k$ has at most as many shiftings as
	$u$ in total and we thus obtain $u^k \in D_{\max}^\Sigma(U)$.  By
	construction, we have
	\[\tfrac{1}{\lambda(I_\ell)}\int_{I_\ell} u(t)\ \d t =
	\tfrac{1}{\lambda(I_\ell)} \sum_{k=1}^s \int_{I^k_\ell} w_k\ \d t
	=\sum_{k=1}^s \tfrac{\lambda(I^k_\ell)}{\lambda(I_\ell)} w_k\] with
	$\nicefrac{\lambda(I^k_\ell)}{\lambda(I_\ell)}\geq 0$ for every
	$k\in \{1,\ldots,s\}$ and $\sum_{k=1}^s
	\nicefrac{\lambda(I^k_\ell)}{\lambda(I_\ell)}=1$. Since the control
	is unchanged on the other intervals $I_i$, $i\neq \ell$, we obtain
	$\Pi(u)=\sum_{k=1}^s
	\nicefrac{\lambda(I^k_\ell)}{\lambda(I_\ell)}\Pi(u^k)$. The
	functions $u^k$ have no shifting in $I_\ell$ so that the number of
	intervals in which at least one of the switches is shifted is at
	most~$m$. According to the induction hypothesis, the
	vectors~$\Pi(u^k)$ can be written as a convex combinations of vectors in~$K$ and consequently, due to $\Pi(u)=\sum_{k=1}^s
	\nicefrac{\lambda(I^k_\ell)}{\lambda(I_\ell)}\Pi(u^k)$, $\Pi(u)$ is
	also a convex combination of vectors in~$K$.
\end{proof}
It is easy to see that Theorem~\ref{thm: dmax01} also extends to the
constraint~$D_{\max}$ defined in~\eqref{eq:Dex}. Indeed, whenever the
constraint set~$D$ is defined by switch-wise constraints as
in~\eqref{eq:Dex}, polyhedricity and integrality can be verified for
each switch individually, in which case~$D_{\max}$ reduces
to~$D_{\max}^\Sigma(\{0,1\})$.

The fact that~$C_{D_{\max}^\Sigma(U),\Pi}$ is a polytope allows, in
principle, to describe it by finitely many linear
inequalities. However, the number of its facets may be exponential
in~$n$ or~$M$, so that a separation algorithm will be needed for the
outer approximation algorithm presented in the companion paper~\cite{partII}. It
depends on the set~$U$ whether this separation problem can be
performed efficiently. E.g., if~$U$ models arbitrary conflicts between
switches that may not be used simultaneously, the separation problem
turns out to be NP-hard, since~$U$ can model the independent set
problem in this case.

Even for~$n=1$ and~$U=\{0,1\}$, the separation problem is
non-trivial. In this case, the set~$K$ defined in Theorem~\ref{thm: dmax01}
consists of all binary sequences~$v_1,\dots,v_M\in\{0,1\}$ such
that~$v_{i-1}\neq v_i$ for at most~$\smax$
indices~$i\in\{2,\dots,M\}$. For the slightly different setting where
$v_1$ is fixed to zero, it is shown in~\cite{BH23} that the
separation problem for~$\conv(K)$ and hence for~$C_{D_{\max},\Pi}$ can
be solved in polynomial time. More precisely, a complete linear
description of~$C_{D_{\max},\Pi}$ is given by~$v\in[0,1]^M$,
$v_1=0$, and inequalities of the form
\begin{equation}\label{eq:alt}
\sum_{j=1}^m (-1)^{j+1} v_{i_j} \leq \Big\lfloor \frac{\smax}{2}
\Big \rfloor\;,
\end{equation} 
where~$i_1,\dots,i_m\in \{2,\dots,M\}$ is an increasing sequence of
indices with~$m-\smax$ odd and $m>\smax$. For given~$\bar
v\in[0,1]^M$, a most violated inequality of the form~\eqref{eq:alt} is
obtained by choosing~$\{i_1,i_3,\dots\}$ as the local maximizers
of~$\bar v$ and~$\{i_2,i_4,\dots\}$ as the local minimizers of~$\bar
v$ (excluding~$1$); such an inequality can thus be computed in~$O(M)$ time.
This separation algorithm is used in Section~\ref{sec: qbounds} to investigate the strength of our convex relaxation.

\subsection{Switching point constraints}\label{ex: combswpoint}

In this section, we focus on the case~$n=1$. 
It is well known that a function $u \in BV(0,T)$ admits a right-continuous representative given by 
$\hat u(t) = c+\mu([0,t])$, $t\in (0,T)$, where $\mu$ is the regular Borel measure on $[0,T]$
associated with the distributional derivative of $u$ and $c\in \R$ a constant. Note that $\hat u$ is unique on~$(0,T)$.
Given $u \in BV(0,T)$ with its right-continuous representative $\hat u$, 
we denote the essential jump set of $u$ by
\[
J_u := \Big\{ t\in (0,T) \colon\; \lim_{\tau \nearrow t} \hat u(\tau) \neq \lim_{\tau \searrow t}\hat{u}(\tau) \Big\}.
\]
In the following, we assume that $u\in BV(0,T)$ always starts with zero. More formally, if $\lim_{\tau \searrow 0}\hat{u}(\tau)=1$, we already count this as one switching from zero to one and add a switching point $t=0$ to $J_u$.
If $J_u$ is a finite set, we denote its cardinality by $|J_u|$. 
For the rest of this section, let~$\sigma \in \N$ be given. 

\begin{definition}
	Let $0 \leq t_1\leq \ldots \leq t_\sigma < \infty$ be given and set 
	\[
	\begin{aligned}[t]
	\eta_\le : \R \to \{0, \ldots, \sigma\}, & &         
	\eta_\le(t) & := |\{i \in \{1, \ldots, \sigma\} \colon \, t_i \le t \}| \\
	\eta_= : \R \to \{0, \ldots, \sigma\}, & &         
	\eta_=(t) & := |\{i \in \{1, \ldots, \sigma\} \colon \, t_i = t \}|    
	\end{aligned}
	\]
	with the usual convention $|\emptyset| = 0$.
	Then we define the function $u_{t_1,\dots,t_\sigma}$ by
	\begin{equation}\label{eq:urepr_def}
	\begin{aligned}
	& u_{t_1,\dots,t_\sigma}\colon [0, T] \to \{0,1\} ,\\
	& u_{t_1,\dots,t_\sigma}(t) := 
	\begin{cases}
	0, & \text{if \,$\eta_\le(t)$ is even},\\
	1 , & \text{if\, $\eta_\le(t)$ is odd}.
	\end{cases}                
	\end{aligned}
	\end{equation}    
\end{definition}

It is easy to verify that $u_{t_1,\dots,t_\sigma}$ is a representative of $u$. Moreover, the function is right-continuous by construction, so that it agrees with the unique right-continuous representative~$\hat{u}\in[u]$ on $(0,T)$.
Now, given any polytope~$P\subseteq \R_+^\sigma$, we define the 
\emph{set of switching point constraints} by
\[ \begin{aligned}
D_P := \{ u \in BV(0,T;\{0,1\}) \colon \; 
& \exists\, 0\leq t_1\leq \cdots\leq t_\sigma < \infty\\ &  \text{ s.t.\ }  (t_1, \ldots, t_\sigma) \in P, \ u_{t_1,\dots,t_\sigma} \in [u] \}\;.
\end{aligned}
\]

\begin{lemma}
	The set $D_P$ satisfies the assumptions in \eqref{eq:D1} and \eqref{eq:D2}.
\end{lemma}

\begin{proof}
	Since $u\in \{0,1\}$ a.e.~in $(0,T)$ and $|J_u| \leq \sigma$  holds
		for all $u\in D_P$ by construction, every $u\in D_P$ satisfies
	$|u|_{BV(0,T)} \leq \sigma$ such that \eqref{eq:D1} is fulfilled.
	
	To verify \eqref{eq:D2}, consider a sequence $\{u^k\}\subset D_P$ with $u^k \to u$ in $L^p(0,T)$. 
	From \eqref{eq:D1} and~\cite[10.1.1(i)]{ATT14}, we deduce $u \in BV(0,T)$.
	Moreover, there is a subsequence, denoted by the same symbol for convenience,
	such that the sequence of representatives~$\{u_{t^k_1,\dots,t^k_\sigma}\}$ converges pointwise 
	almost everywhere in $(0,T)$ to $u$. This yields $u\in \{0,1\}$ a.e.~in $(0,T)$. 
	Furthermore, as a polytope, $P$ is compact by definition, so that
	there is yet another subsequence such that
	$t^k := (t_1^k, \ldots, t_\sigma^k)$ converges to $\bar t\in \R^\sigma$ 
	with $0 \leq \bar t_1 \leq \ldots \leq \bar t_\sigma < \infty$ and $\bar t \in P$.
	The mapping $$P\ni(t_1,\ldots,t_\sigma) \mapsto u_{t_1.\ldots,t_\sigma}\in L^p(0,T)$$ is continuous, which can be seen as follows. If $\{(t_1^k, \ldots, t_\sigma^k)\}_{k\in \N}\subseteq P$ converges to some $\bar t\in \R^\sigma$, then for every $t\in (0,T)\setminus \{\bar t_1,\ldots,\bar t_\sigma\}$ it is clear that 
		$$|\{i\in\{1,\ldots,\sigma\}\colon t_{i}^k\leq t\}|=|\{i\in\{1,\ldots,\sigma\}\colon \bar{t}_{i}\leq t\}|$$
		holds for $k$ sufficient large, so that 
		$u_{t^k_1,\dots,t^k_\sigma}(t)\to u_{\bar t_1,\ldots,\bar t_\sigma}(t)$  for $k\to \infty$ follows by the definition of the representatives in~\eqref{eq:urepr_def}. Consequently,  $\{u_{t^k_1,\dots,t^k_\sigma}\}_{k\in \N}$ converges pointwise 
		almost everywhere to~$ u_{\bar t_1,\ldots,\bar t_\sigma}$ in~$(0,T)$. By Lebesgue's dominated convergence theorem, see, \eg~\cite[Lemma~3.25]{ALT16}, $\{u_{t^k_1,\dots,t^k_\sigma}\}_{k\in \N}$ then also converges strongly to~$ u_{\bar t_1,\ldots,\bar t_\sigma}$  in $L^p(0,T)$. Thus, we have $$u=\lim\limits_{k\to \infty} u^k=\lim\limits_{k\to \infty}u_{t^k_1,\dots,t^k_\sigma} = u_{\bar t_1,\ldots,\bar t_\sigma} \quad \text{in } L^p(0,T),$$ which gives $u\in D(P)$.
\end{proof}

\begin{theorem}\label{thm: combswpoint}
	For any\/ $\Pi$ as in~\eqref{eq:pi}, the
	set~$C_{D_P,\Pi}$ is a polytope in~$\R^M$.
\end{theorem}
\begin{proof}
	Let~$0=s_0< s_1< \dots< s_{r-1}<s_r=\infty$ include all end points
	of the intervals~$I_{1}, \ldots, I_{M}$ defining~$\Pi$. Let~$\Phi$ be the set of all
	maps~$\varphi\colon\{1,\dots,\sigma\}\rightarrow\{1,\dots,r\}$. Then we have
	\begin{equation}\label{eq:decomp}
	\big\{ (t_1,\dots,t_\sigma)\in P\colon t_1\le\dots\le
	t_\sigma\big\} =\bigcup_{\varphi\in\Phi} P_\varphi
	\end{equation}
	with
	\[P_\varphi:=\big\{ (t_1,\dots,t_\sigma)\in P\colon t_1\le\dots\le
	t_\sigma,\;s_{\varphi(i)-1}\le t_i\le
	s_{\varphi(i)}\;\forall i=1,\dots,\sigma \big\}\;.\] Now each
	set~$P_\varphi$ is a (potentially empty) polytope. Moreover, by
	construction, the function~$P_\varphi\ni
	(t_1,\dots,t_\sigma)\mapsto\Pi(u_{t_1,\dots,t_\sigma})\in\R^M$ is
	linear, since
	\[
	\Pi(u_{t_1,\dots,t_\sigma})_j
	=  \tfrac{1}{\lambda(I_j)}\int_{I_j} u_{t_1,\dots,t_\sigma}(t)\,\d t\\
	=  \tfrac{1}{\lambda(I_j)}\sum_{\tiny\substack{i\in\{1,\dots,\sigma+1\}\\\text{ even}}}\int_{I_j}  \chi_{[t_i-t_{i-1}]}\,\d t
	\] for~$j=1,\dots,M$, where we set~$t_0:=0,t_{\sigma+1}:=\infty$,
	and $\int_{I_j} \chi_{[t_i-t_{i-1}]}\,\d t$ is linear 
	in~$t_i$ and~$t_{i-1}$ for a fixed assignment~$\varphi$.  It
	follows from~\eqref{eq:decomp} that~$\Pi(D_P)$ is a finite union of
	polytopes and hence its convex hull~$C_{D_P,\Pi}$ is a polytope
	again.
\end{proof}

An important class of constraints of type~$D_P$ are the minimum
dwell-time constraints. For a given minimum dwell time~$s>0$, it is
required that the time elapsed between two switchings is at
least~$s$. This implies, in particular, that the number of such
switchings is bounded by~$\sigma:=\lceil T/s\rceil$. We thus consider
the constraint
\[
\begin{aligned}
D_s:=\big\{ u\in BV(0,T)\colon &\exists \ t_1,\ldots,t_\sigma\geq 0 \\  &\text{ s.t.\ } t_{j}-t_{j-1}\ge s\; \forall \,j = 2, \dots,\sigma, \; u_{t_1,\dots,t_\sigma} \in [u] \big\}.
\end{aligned}
\]
By Theorem~\ref{thm: combswpoint}, the set~$C_{D_s,\Pi}$ is a polytope
in~$\R^M$. However, it is not a 0/1-polytope in general. As an
example, consider the time horizon~$[0,3]$ with
intervals~$I_j:=[j-1,j]$ for each~$j=1,2,3$ and let~$s=\tfrac 32$. Then it
is easy to verify that~$C_{D_s,\Pi}$ has several fractional vertices,
\eg the vector~$(0,1,\tfrac 12)^\top$, being the unique optimal
solution when minimizing~$(1,-1,\tfrac 12)^\top x$ over~$x\in C_{D_s,\Pi}$. Nevertheless, the separation problem for~$D_s$ can be
solved efficiently, as we will show in the following. Our approach is
thus well-suited to deal with minimum dwell time constraints as well.

In order to show tractability, we first argue that it is enough to consider as switching points the finitely many points in the set
\[
S:=[0,T]\cap \Big(\Z s+\big(\{0,T\}\cup\{a_i,b_i\colon
i=1,\dots,M\}\big)\Big)
\] where~$I_i=[a_i,b_i]$ for~$i=1,\dots,M$. The
set~$S$ thus contains all end points of the intervals~$I_1,\dots,I_M$
and $[0,T]$ shifted by arbitrary integer multiples of~$s$, as long as
they are included in~$[0,T]$. Clearly, we can compute~$S$
in~$O(M\sigma)$ time. Let~$\tau_1\dots,\tau_{|S|}$ be the elements
of~$S$ sorted in ascending order.
\begin{lemma}\label{lem:dwell}
	Let~$v$ be a vertex of~$C_{D_s,\Pi}$. Then there exists~$u\in D_s$ with~$\Pi(u)=v$ such that~$u$ switches only in~$S$.
\end{lemma}
\begin{proof}
	Choose~$c\in\R^M$ such that~$v$ is the unique minimizer of~$c^\top
	v$ with~$v\in C_{D_s,\Pi}$. Moreover, choose any~$u\in D_s$
	with~$\Pi(u)=v$ and let~$t_1,\dots,t_\sigma$ be the switching points of~$u$, i.e., let $0\le t_1\le\dots\le t_\sigma<\infty$ such that $u_{t_1,\dots,t_\sigma}\in [u]$. For the following, define
	$$S'_j:=\{t_\ell\mid
	\ell\in\{1,\dots,\sigma\},\ t_\ell-t_j=s(\ell-j)\}$$
	for~$j=1,\dots,\sigma$.
	
	Assume
	first that $t_j\in(a_i,b_i)\setminus S$ for
	some~$i\in\{1,\dots,M\}$ and some~$j\in\{1,\dots,\sigma\}$. By
	definition of~$S$, all switching points having minimal distance
	to~$t_j$ do not belong to~$S$ as
	well, i.e., $S'_j\cap S=\emptyset$. Hence all points in $S_j'$ can be shifted simultaneously by some
	small enough~$\varepsilon>0$, in both directions, maintaining
	feasibility with respect to~$D_s$ and without any of these points
	leaving or entering any of the intervals~$I_1,\dots,I_M$ and
	$[0,T]$. This shifting thus changes the value of~$c^\top\Pi(u)$
	linearly, as seen in the proof of~Theorem~\ref{thm: combswpoint}, which is
	a contradiction to unique optimality of~$v$.
	
	We have thus shown that any~$u\in D_s$ with~$\Pi(u)=v$ must have all
	switching points either in~$S$ or outside of any interval~$I_i$. So
	consider some~$u\in D_s$ with~$\Pi(u)=v$, defined by switching
	points~$t_1,\dots,t_\sigma$ as above, and let~$t_j\not\in
	S$ be any switching point of~$u$ not belonging to any
	interval~$I_i$. By shifting all switching points in~$S'_j$ simultaneously
	to the left until~$S_j'\cap S\neq\emptyset$, taking into account that the set~$S_j'$ may increase when~$t_j$ decreases, we obtain another
	function~$u'\in D_s$. By construction of~$S$, no shifting point is
	moved beyond the next point in~$S$ to the left of its original
	position. In particular, none of the shifting points being moved
	enters any of the intervals~$I_i$, so that we
	derive~$\Pi(u')=\Pi(u)=v$, but~$u'$ has strictly less switching
	points outside of~$S$ than~$u$. By repeatedly applying the same
	modification, we eventually obtain a function projecting to~$v$ with
	switching points only in~$S$.
\end{proof}
\begin{theorem}\label{thm:dwell}
	One can optimize over~$C_{D_s,\Pi}$ (and hence also separate
	from~$C_{D_s,\Pi}$) in time polynomial in~$M$ and~$\sigma$.
\end{theorem}
\begin{proof}
	By Lemma~\ref{lem:dwell}, it suffices to optimize over the
	projections of all~$u\in D_s$ with switchings only in~$S$. This
	can be done by a simple dynamic programming approach:
	given~$c\in\R^M$, we can compute the optimal value
	\[c^*(t,b):=\min\;c^\top \Pi(u\cdot\chi_{[0,t]})\; \text{
		s.t. }u\in D_s,\;\lim\limits_{\tau\searrow t}\hat{u}(\tau)=b\text{ if }t<T\] for~$b\in\{0,1\}$
	recursively for all~$t\in S$.
	Starting with~$c^*(\tau_1,b)=0$, we obtain
	\[
	c^*(\tau_j,b)=\min\begin{cases}
	\begin{array}{ll}
	c^*(\tau_{j-1},b)+c^\top\Pi(b\chi_{[\tau_{j-1},\tau_{j}]})\\
	c^*(\tau_{j}-s,1-b)+c^\top\Pi((1-b)\chi_{[\tau_{j}-s,\tau_{j}]}), & \text{if }\tau_j\ge s\\
	c^\top\Pi((1-b)\chi_{[0,\tau_j]}), & \text{if }\tau_j< s, b=1
	\end{array}
	\end{cases}
	\]
	for~$j=1,\dots,|S|$.  The desired optimal value
	is~$\min\{c^*(T,0),c^*(T,1)\}$ then, and a corresponding optimal
	solution can be derived easily.
\end{proof}
Note that~$\sigma$ is not polynomial in the input size in general, but
only pseudopolynomial, if~$T$ and~$s$ are considered part of the input.

In practice, it is necessary to design an explicit separation
algorithm for~$C_{D_s,\Pi}$ instead of using the theoretical
equivalence between separation and optimization. This might be
possible by generalizing the results presented in~\cite{lee04}. In
fact, in the special case that~$[0,T]$ is subdivided into
intervals~$I_1,\dots,I_M$ of the same size and this size is a divisor
of~$s$, it follows from Lemma~\ref{lem:dwell} that~$C_{D_s,\Pi}$ agrees
with the min-up/min-down polytope investigated in~\cite{lee04}. In
this case,~$C_{D_s,\Pi}$ is a 0/1-polytope and a full linear
description, together with an exact and efficient separation
algorithm, is given in~\cite{lee04}. It might be possible to obtain
similar polyhedral results for~$C_{D_s,\Pi}$ also in the general case.
We leave this as future work.

To conclude this section, we note that the latter results can easily be transferred to a situation where the minimum dwell time after switching up is different from the minimum dwell time after switching down, which is often considered in the literature. More generally, we may consider any~$\bar s\in\R_+^\sigma$ and define
\[
\begin{aligned}
D_{\bar s}:=\big\{ u\in BV(0,T)\colon &\exists \ t_1,\ldots,t_\sigma\geq 0 \\  &\text{ s.t.\ } t_1\ge \bar s_1,\;t_{j}-t_{j-1}\ge \bar s_j\; \forall \,j = 2, \dots,\sigma,\; u_{t_1,\dots,t_\sigma} \in [u] \big\}.
\end{aligned}
\]
In order to generalize the results obtained for~$D_s$, it suffices to replace the set~$S$ used above by the set
$$\bar S:=[0,T]\cap \Big(\{0\}\cup\big\{\pm\textstyle\sum_{j=\ell_1}^{\ell_2}\bar s_j\mid 1\le \ell_1\le \ell_2\le \sigma\big\}+\big(\{0,T\}\cup\{a_i,b_i\colon
i=1,\dots,M\}\big)\Big)\;,$$
which can be computed in~$O(M\sigma^2)$ time. Using~$\bar S$ in place of~$S$ and following the same reasoning, both Lemma~\ref{lem:dwell} and Theorem~\ref{thm:dwell} also hold for~$D_{\bar s}$.

\section{Numerical evaluation of bounds} \label{sec: qbounds}

In this section, we test the quality of our outer description of the convex hull
and, in particular, the strength of the resulting lower bounds. For this, we concentrate on the
case of a single switch with an upper bound~$\smax$ on the number of
switchings, \ie we consider
\[
D := \big\{ u \in BV(0,T)\colon \; u(t) \in \{0,1\} \text{ f.a.a.\ } t \in (0,T),\; |u|_{BV(0,T)} \leq \smax \big\}.
\]
However, we assume that $u$ is fixed to zero before the time horizon,
so that we count it as a shift if $u$ is $1$ at the beginning. Moreover, we consider exemplarily a square domain $\Omega=[0,1]^2$, the
end time $T=2$, the upper bound $\smax=2$ on the number of switchings
and the form function $\psi$ as well the desired state $y_\textup{d}$
given as
\[
\begin{aligned}
\psi(x)&:=12\pi^2 \exp(x_1+x_2)\sin(\pi\,x_1)\sin(\pi\,x_2) \\ 
y_\textup{d}(t,x)&:=2\pi^2\,\max(\cos(2\pi\, t),0)\sin(\pi\,x_1)\sin(\pi\,x_2). \\ 
\end{aligned}
\]
We always choose $\alpha=0$, so that the computed bounds are
not deteriorated by the Tikhonov term.

For the discretization of the optimal control problem, we use the \textsc{DUNE}-library~\cite{SAN21}. To obtain exact optimal solutions for comparison, we use the MINLP solver \textsc{Gurobi~9.1.2}~\cite{gurobi} for solving the
discretized problem. The source code is part of the implementation at~\url{https://github.com/agruetering/dune-MIOCP}.
The spatial discretization uses a standard Galerkin method with continuous and piecewise linear functionals. For the state~$y$ and the desired temperature $y_\textup{d}$ we also use continuous and piecewise linear functionals in time, while the temporal discretization for the controls chooses piecewise constant functionals.
The
BV-seminorm condition then simplifies to
\begin{equation}\label{eq:semidisc}
u_0+\sum_{i=1}^{N_t-1} |u_i-u_{i-1}|\leq\smax\;,
\end{equation}
where the term~$u_0$ is added in order to count a shift if $u_0=1$.
We linearize~\eqref{eq:semidisc} by introducing $N_t-1$ additional real variables $z_i$ expressing the absolute values $|u_i-u_{i-1}|$. More precisely, we require~$z_i\ge u_i-u_{i-1}$ and~$z_i\ge u_{i-1}-u_i$ and use the linear constraint $u_0+\sum_{i=1}^{N_t}z_i\le \smax$ instead of~\eqref{eq:semidisc}. The naive convex relaxation now replaces the binarity constraint~$u_i\in\{0,1\}$ with~$u_i\in[0,1]$
for~$i=0,\dots,N_t-1$. For the tailored convexification presented in this
paper, we instead omit the constraint~\eqref{eq:semidisc} and iteratively add
a most violated cutting plane for~$C_{D,\Pi}$, where the intervals $I_1,\ldots,I_M$ for the projection are the ones given by the discretization in time, until the
relative change of the bound is less than $0.1\%$ in three successive
iterations. To the best of our knowledge, there is no standard procedure for solving the convexified control problems with additional linear control constraints arising at this point, we thus also use \textsc{Gurobi~9.1.2}~\cite{gurobi} for this.

We investigate the bounds for a sequence of discretizations with
various numbers~$N_t$ of time intervals and uniform spatial
triangulations of~$\Omega$ with $N_x\times N_x$ nodes. \textsc{Gurobi}
is run with default settings except that the parallel mode is switched
off for better comparison and the dual simplex method is used due to
better performance. All computations have been performed on a 64bit Linux system with an Intel Xeon E5-2640 CPU @ 2.5 GHz and $32$ GB RAM.

\bgroup \def\arraystretch{1.5}
\begin{table}[htb]
	\centering
	\begin{scriptsize}
		\begin{tabular}{rrrrrrrrrrr}
			\hline 
			$N_x$ & $N_t$ & \multicolumn{2}{c}{MINLP} & \multicolumn{2}{c}{naive rel.} &  \multicolumn{4}{c}{tailored convexification} &\\
			\cmidrule(lr){3-4}\cmidrule(lr){5-6}\cmidrule(lr){7-11} &  & Obj & Time (s) & Obj & Gap & Obj & \#Cuts & \#Ex &  Gap & Filled gap \\
			\hline
			10& 20 & 13.69 &     4.38 & 8.41 & 38.60 \% & 9.67 &  21 & 5&29.39 \% & 23.85 \% \\
			& 40 & 12.76 &    39.51 & 7.39 & 42.09 \% & 9.03 &  56 & 16&29.22 \% & 30.59 \% \\
			& 60 & 12.51 &   152.44 & 7.29 & 41.71 \% & 8.86 & 108 & 30&29.19 \% & 30.01 \% \\
			& 80 & 12.50 &   465.19 & 7.28 & 41.75 \% & 8.53 & 143 & 67&31.74 \% & 23.98 \% \\
			& 100 & 12.54 &   663.20 & 7.25 & 42.18 \% & 7.95 & 136 & 88&36.62 \% & 13.18 \% \\
			\hline
			15& 20 & 13.69 &    32.07 & 8.38 & 38.79 \% & 9.67 &  21 & 5&29.40 \% & 24.20 \% \\
			& 40 & 12.76 &   272.97 & 7.38 & 42.13 \% & 9.08 &  52 & 16&28.83 \% & 31.57 \% \\
			& 60 & 12.51 &  1183.68 & 7.29 & 41.75 \% & 8.71 &  91 & 28&30.35 \% & 27.30 \% \\
			& 80 & 12.50 &  2837.12 & 7.28 & 41.78 \% & 8.35 & 117 & 60&33.19 \% & 20.57 \% \\
			& 100 & 12.54 &  4686.54 & 7.25 & 42.22 \% & 7.79 & 131 & 93&37.91 \% & 10.20 \% \\
			\hline
			20& 20 & 13.69 &   109.08 & 8.37 & 38.85 \% & 9.69 &  23 & 5&29.24 \% & 24.73 \% \\
			& 40 & 12.76 &  1305.88 & 7.38 & 42.14 \% & 8.96 &  59 & 20&29.75 \% & 29.40 \% \\
			& 60 & 12.51 &  5147.66 & 7.29 & 41.76 \% & 8.57 &  86 & 35&31.52 \% & 24.53 \% \\
			& 80 & 12.50 & 15185.22 & 7.28 & 41.78 \% & 8.30 & 123 & 62&33.62 \% & 19.53 \% \\
			& 100 & 12.54 & 19550.01 & 7.25 & 42.23 \% & 8.00 & 153 & 91 &36.20 \% & 14.27 \% \\
			\hline
		\end{tabular} 
	\end{scriptsize}
	\caption{Comparison of naive and tailored convexification.}\label{tab:gurobi}
\end{table}
\egroup

The results are presented in Table~\ref{tab:gurobi}. For given choices
of~$N_t$ and~$N_x$, we report the objective values (Obj) obtained by
the exact approach and the two relaxations. We emphasize that, for a
given optimal solution of the respective problem, we recalculate the
objective value with a much finer discretization, choosing $N_t=200$
and $N_x=100$. In particular, the bounds do not necessarily behave
monotonously. It can be seen from the results that the new bounds are clearly
stronger than the naive bounds. In the last column (Filled gap), we
state how much of the gap left open by the naive relaxation is closed
by the new relaxation. We also state how many cutting planes are
computed altogether (\#Cuts) and how many of them are needed to obtain
at least the same bound as the naive relaxation (\#Ex). The main
message of Table~\ref{tab:gurobi} is that our new approach yields better
bounds than the naive approach even after adding relatively few
cutting planes. Additionally, the naive relaxation includes inequality
constraints involving the BV-seminorm, such that its solution is very
challenging in practice.

For the exact approach, we also state the time (in seconds) needed
for the solution of the problem (Time). It is obvious from the results
that only very coarse discretizations can be considered when using a
straightforward MINLP-based approach. In the companion paper~\cite{partII},
we thus develop a tailored outer approximation algorithm based on the the convex hull description in \eqref{eq:convapprox} in order  to compute the dual bounds of our convex relaxation of~\eqref{eq:optprob} more efficiently.

\fontsize{9}{10.5}\selectfont
\bibliographystyle{siam}
\bibliography{reference}
\end{document}